\documentclass[reqno]{amsart}
\usepackage[hidelinks]{hyperref}
\usepackage{amssymb}
\usepackage{enumerate}
\usepackage{cite}
\usepackage{xcolor}
\begin{document}

\title[Rayleigh-Stokes equations]
{On global solvability and regularity for generalized Rayleigh-Stokes equations with history-dependent nonlinearities}
\author[T.D. Ke, N.N. Thang]{Tran Dinh Ke \& Nguyen Nhu Thang$^*$}
\thanks{* Corresponding author. Email: thangnn@hnue.edu.vn (N.N.Thang)}
\begin{abstract}
We are concerned with the initial value problem governed by generalized Rayleigh-Stokes equations, where the nonlinearity depends on history states and takes values in Hilbert scales of negative order. The solvability and H\"older regularity of solutions are proved by using fixed point arguments and embeddings of fractional Sobolev spaces. An application to a related inverse source problem is given.
\end{abstract}
\maketitle
\numberwithin{equation}{section}
\newtheorem{theorem}{Theorem}[section]
\newtheorem{lemma}[theorem]{Lemma}
\newtheorem{proposition}[theorem]{Proposition}
\newtheorem{corollary}[theorem]{Corollary}
\newtheorem{definition}{Definition}[section]
\newtheorem{remark}{Remark}[section]
\newtheorem{example}{Example}[section]

\section{Introduction}
Let $\Omega\subset \mathbb R^d$ be a bounded domain with smooth boundary $\partial\Omega$. We consider the following problem 
\begin{align}
\partial_t u - (1+D_t^{\{m\}})\Delta u & = f(u, \mathcal H u)\text{ in } \Omega, t\in (0,T),\label{e1}\\
u & = 0 \text{ on } \partial\Omega,\; t\ge 0,\label{e2}\\
u(0) & = \xi \text{ in } \Omega, \label{e3}
\end{align}
where $f$ is a given function, $\partial_t=\frac{\partial}{\partial t}$,  $D_t^{\{m\}}$ is the nonlocal derivative of Riemann-Liouville type defined by
$$
D_t^{\{m\}} v(t) = \frac{d}{dt}\int_0^t m(t-s)v(s)ds,
$$
with kernel $m\in L^1_{loc}(\mathbb R^+)$, $\mathcal H$ is the convolution operator given by $$\mathcal Hv(t)=\int_0^t \ell (t-\tau)v(\tau)d\tau,\; \ell \in L^1(0,T),$$ which tells us that $f$ depends on the history state of the system.

The proposed system is a general model for some problems studied in literature. Indeed, in the case $m$ is a constant, \eqref{e1} is of classical diffusion type. If $m$ is a regular function, e.g. $m\in C^1(\mathbb R^+)$ then one gets
\begin{equation*}
\partial_t u - (1+m_0)\Delta u -\int_0^t m_1(t-s)\Delta u(s) ds = f,
\end{equation*}
with $m_0=m(0)$ và $m_1(t)=m'(t)$, which is a nonclassical diffusion equation. Let $m(t) = m_0 t^{-\alpha}/\Gamma(\alpha)$ with $m_0>0$, then we see that \eqref{e1} is a Rayleigh-Stokes equation, i.e.
\begin{equation}\label{RS-eq}
\partial_t u - (1+m_0\partial_t^\alpha)\Delta u  = f.
\end{equation}
The constitution of the last equation was formed in \cite{FJFV09,STZM06} to describe the behavior of second-grade fluids. It should be mentioned that, some numerical schemes for  Rayleigh-Stokes equations were developed in \cite{Bazh15,Bi18,Chen13,Chen08}. On the other hand, analytical representations for solution of \eqref{RS-eq} in linear case were obtained in \cite{Khan09,STZM06}, and recently, the regularity for nonlinear Rayleigh-Stokes equations has been established in \cite{Lan20,Luc21,ZW19}. For more studies related to \eqref{RS-eq}, we refer the reader to \cite{Luc19,NLATZ,Tuan19}, where the terminal value problem was carried out. 

In this work, we are interested in the solvability and regularity analysis for problem \eqref{e1}-\eqref{e3} in the circumstance that the nonlinearity function involves the history state expressed by $\mathcal H$. The appearance of $\mathcal Hu$ comes from, e.g. the inverse source problem as presented in the last section. In another way, this factor may arise from control problems, where the feedback requires some history information of the system. One will find that, the term $\mathcal Hu$ causes some difficulties for analyzing regularity of solutions. As an important feature of our study, it is noted that the nonlinearity function $f$ is allowed to take weak values, i.e. $f(u,\mathcal Hu)$ may belong to dual of fractional Sobolev spaces. This enables us to consider the case when $f$ contains polymonial or gradient terms, which have connections with practical applications. This also extends the recent results established in \cite{KT2022,Lan20,Luc21}.

The rest of this note is as follows. In the next section, we first recall some notions and facts related to Hilbert scales and fractional Sobolev spaces. Additionally, a representation for the solution of linear problem will be shown together with some essential estimates for resolvent operator. Section 3 is devoted to proving the main results, including the global existence and H\"older regularity of solutions. It is worth noting, in particular, that the H\"older regularity result in our work can not be obtained by the technique used in \cite{KT2022}, since the differentiability of resolvent operator is unavailable on dual spaces. In order to overcome this impediment, we construct a regular closed subset of solution space, which is invariant under the solution operator, and make use of fixed point arguments. The last section shows an application of the obtained results, where we demonstrate that an inverse source problem governed by Rayleigh-Stokes equations is solvable by transforming it to a prototype of problem \eqref{e1}-\eqref{e3}.

\section{Preliminaries}
\subsection{Functional spaces}
Let $-\Delta$ be the Laplacian associated with the homogenuous Dirichlet boundary condition. Then one has a sequence of eigenfunctions $\{e_n\}$ of $-\Delta$ which forms an orthonormal basis of $L^2(\Omega)$, and we have the following representation 
$$-\Delta v = \sum _{n=1}^\infty \left(\lambda_n \int\limits_\Omega v(x)e_n(x)dx\right)e_n,$$
where $\lambda_n>0$ is the eigenvalue corresponding to the eigenfunction  
$e_n$. For $\varrho\ge 0$, we define the following functional space
$$\mathbb H^\varrho:=\left\{\varphi\in L^2(\Omega) \bigg| \|\varphi\|^2_{\mathbb H^\varrho} := \sum_{n=1}^\infty \left|\lambda_n^{\frac \varrho 2}\int\limits_\Omega \varphi(x)e_n(x)dx  \right |^2<\infty  \right\}.$$
Let $\mathbb H^{-\varrho}$ denote the dual space of $\mathbb H^\varrho$ with respect to the dual pair $\langle \cdot,\cdot \rangle_{-\varrho,\varrho} $ on $\mathbb H^{-\varrho} \times \mathbb H^\varrho$  with the norm
$$  \|\varphi\|^2_{\mathbb H^{-\varrho}} := \sum_{n=1}^\infty \left|\lambda_n^{-\frac \varrho 2} \langle \varphi,e_n \rangle_{-\varrho,\varrho} \right |^2<\infty.$$
Clearly, $\mathbb H^{\varrho_2} \hookrightarrow \mathbb H^{\varrho_1}$ and $\mathbb H^{-\varrho_1} \hookrightarrow \mathbb H^{-\varrho_2}$ with $\varrho_2\ge \varrho_1\ge 0$. The family of Hilbert spaces $\mathbb H^\varrho, \varrho\in\mathbb R$, is said to be the Hilbert scales. 

The fractional Laplacian $(-\Delta)^\gamma$, $\gamma\ge 0$, is defined as follows
\begin{align*}
(-\Delta)^\gamma: \; \mathbb H^{\gamma} &\to \mathbb H^{-\gamma};\\
 (-\Delta)^\gamma \varphi &=  \sum_{n=1}^\infty \lambda_n^\gamma \left(\int\limits_\Omega \varphi(x)e_n(x)dx \right)e_n.
\end{align*}
Then, $\|(-\Delta)^\gamma \varphi\|_{\mathbb H^{-\gamma}}  = \|\varphi\|_{\mathbb H^{\gamma}}$. 

We now recall the notion of the fractional Sobolev spaces (see, e.g. \cite{Dem, Di} for details). For $r \in(0,1)$ and $p \in[1,+\infty)$, the functional space 
$$
W^{r, p}(\Omega):=\left\{u \in L^{p}(\Omega): \int_{\Omega} \int_{\Omega} \frac{|u(x)-u(y)|^{p}}{|x-y|^{d+p r}} d x d y<\infty\right\}
$$
is called the fractional Sobolev space with the norm
$$
\|u\|_{W^{r, p}(\Omega)}=\left(\int_{\Omega}|u|^{p} d x+\int_{\Omega} \int_{\Omega} \frac{|u(x)-u(y)|^{p}}{|x-y|^{d+p r}} d x d y\right)^{1 / p}
$$
For $r \geq 1$, denote by $[r]$ the integral part of $r$, we define $$W^{r, p}(\Omega):=\left\{u \in W^{[r], p}(\Omega): D^{\eta} u \in W^{r-[r], p}(\Omega), \forall |\eta| \leq[r]\right\}$$ which is a Banach space with the norm
$$
\|u\|_{W^{r, p}(\Omega)}=\left(\|u\|_{W^{[r], p}}^{p}(\Omega)+\sum_{|\eta|=[r]} \int_{\Omega} \int_{\Omega} \frac{\left|D^{\eta} u(x)-D^{\eta} u(y)\right|^{p}}{|x-y|^{d+p(r-[r])}} d x d y\right)^{1 / p}.
$$
Denote
$$
W_{0}^{r, p}(\Omega):=\overline{C_{c}^{\infty}(\Omega)}^{W^{r, p}(\Omega)}, H^{r}(\Omega):=W^{r, 2}(\Omega), H_{0}^{r}(\Omega):=W_{0}^{r, 2}(\Omega).
$$
Assume that $\Omega$  is a domain with sufficiently smooth boundary  such that $C_{c}^{\infty}(\Omega)$ is dense in $H^{r}(\Omega)$ with $0<r<1 / 2$, then  $H_{0}^{r}(\Omega)=H^{r}(\Omega)$ (see \cite[Corollary 8.10.1]{Bha}). It follows from \cite{BSV15} that
\begin{equation}\label{hr}
\mathbb{H}^{r}= \begin{cases}H_{0}^{r}(\Omega), & 0 \leq r<1 / 2, \\ H_{00}^{1 / 2}(\Omega) \varsubsetneqq H_{0}^{1 / 2}(\Omega), & r=1 / 2, \\ H_{0}^{r}(\Omega), & 1 / 2<r \leq 1, \\ H_{0}^{1}(\Omega) \cap H^{r}(\Omega), & 1<r \leq 2,\end{cases}
\end{equation}
where $H_{00}^{1 / 2}(\Omega)$ is the Lions-Magenes space determined by
$$
H_{00}^{1 / 2}(\Omega)=\left\{u \in H^{1 / 2}(\Omega): \int_{\Omega} \frac{|u(x)|^{2}}{(\operatorname{dist}(x, \partial \Omega))^{2}} d x<\infty\right\}.
$$
Then the following lemma is a direct consequence of \eqref{hr}.

\begin{lemma}\label{lm:em} Denote by $H^{-r}(\Omega)$ the dual space of $H_{0}^{r}(\Omega)$ with $r \geq 0$. If $0 \leq r \leq r^{\prime} \leq 2$, then
	$$
	\mathbb{H}^{r^{\prime}} \hookrightarrow \mathbb{H}^{r} \hookrightarrow H^{r}(\Omega) \hookrightarrow L^{2}(\Omega) \hookrightarrow H^{-r}(\Omega) \hookrightarrow \mathbb{H}^{-r} \hookrightarrow \mathbb{H}^{-r^{\prime}}.
	$$
\end{lemma} 
The following lemma represents embeddings between fractional Sobolev spaces.
\begin{lemma}\cite[Theorem 8.12.6]{Bha}\label{sob_emb}
Let $1\le p,p'\le \infty, 0\le r,r' <\infty$ and $r'-\frac{d}{p'} \ge r-\frac{d}{p}$. Then,
$$W^{r',p'}(\Omega) \hookrightarrow W^{r,p}(\Omega).$$ 
\end{lemma}
Using Lemma \ref{lm:em}  and  \ref{sob_emb}, we have the following embeddings.
\begin{lemma}\label{lm:em1} We have
	\begin{enumerate}[a)]
		\item $L^{p}(\Omega) \hookrightarrow H^{r}(\Omega) \hookrightarrow \mathbb{H}^{r}$ if $\left\{-\frac{d}{2}<r \leq 0, p \geq \frac{2 d}{d-2 r}\right\}$.
		\item $\mathbb{H}^{r} \hookrightarrow H^{r}(\Omega) \hookrightarrow L^{p}(\Omega)$ if
		$
		\left\{0 \leq r<\frac{d}{2}, 1 \leq p \leq \frac{2 d}{d-2 r}\right\}.
		$
	\end{enumerate}	
\end{lemma} 
\subsection{Representation of solutions to linear problem}
In order to investigate problem \eqref{e1}-\eqref{e3}, we assume the following hypothesis:
\begin{itemize}
\it
\item[(\textbf{M})] The function $m\in L^1_{loc}(\mathbb R^+)$ is nonnegative such that  $a(t): = 1 + m(t)$ is completely positive.
\end{itemize}
Recall that the real valued function  $a$ is said to be completely positive if the solutions to the following integral equations
\begin{align}
& s(t) + \theta \int_0^t a(t-\tau)s(\tau)d\tau = 1,\; t\ge 0,\label{ire1}\\
& r(t) + \theta \int_0^t a(t-\tau)r(\tau)d\tau = a(t),\; t>0,\label{ire2}
\end{align}
are nonnegative for each $\theta>0$. The theory of completely positive functions can be found in  \cite{CN81,Pruss}. An equivalent condition for $a$ to be completely positive is as follows:
\begin{enumerate}
\item[(PC)] \it There exists a nonincreasing and nonnegative function  $k\in L^1_{loc}(\mathbb R^+)$ and $\epsilon\ge 0$ such that
$$
\epsilon a + k*a = 1\text{ on } (0,\infty).
$$
\end{enumerate}
In the present work, we assume that  $m$ is unbounded on $\mathbb R^+$, which implies  $\epsilon=0$. The condition (PC) is satisfied if $m$ is completely monotone, i.e. $(-1)^n m^{(n)} (t) \ge 0$ on $(0,\infty)$, for all $n\in\mathbb N$ (see \cite{Miller}). 

First we consider the relaxation problem:
\begin{align}
\omega'(t) + \lambda (1+D_t^{\{m\}}) \omega (t) & = 0,\; t> 0,\label{re1}\\
\omega (0) & = 1,\label{re2}
\end{align}
where the unknown $\omega$ is a scalar function, $\lambda$ is a positive parameter. Integrating both sides of \eqref{re1} over $(0,t)$, we obtain
\begin{align}
\omega(t) + \lambda \int_0^t (1+m(t-\tau))\omega(\tau) d\tau=1, \label{re3}
\end{align}
which is just \eqref{ire1} with $\lambda =\theta$.

We denote by $\omega(t,\lambda)$ the solution of \eqref{re3} to emphasize the dependence of  $\omega$ on the parameter $\lambda$. We list some properties of  $\omega$ in the following proposition. 
\begin{proposition}\label{pp-relax-func}
Let $\omega$ be the solution to \eqref{re1}-\eqref{re2}. Then
\begin{enumerate}
\item[(a)] $\omega$ is nonincreasing on  $\mathbb R^+$ and
\begin{align*}
0<\omega(t,\lambda)\le \frac{1}{1+\lambda\int_0^t (1+m(\tau))d\tau},\; \forall t\ge 0,\;\lambda>0.
\end{align*}
\item[(b)] The following estimate holds
\begin{align*}
\int_0^t \omega(\tau,\lambda)d\tau\le \lambda^{-1}(1-\omega(t,\lambda)),\;\forall t\ge 0, \lambda>0.
\end{align*}
\item[(c)] For each $t>0$, the function $\lambda\mapsto \omega(t,\lambda)$ is nonincreasing.
\item[(d)] The function $v(t) = \omega(t,\lambda)v_0 + \int_0^t\omega(t-\tau,\lambda)g(\tau)d\tau$ is a solution to the problem
\begin{align*}
v'(t) + \lambda (1+D^{\{m\}}_t)v(t) &= g(t),\\
v(0)&=v_0.
\end{align*}
\end{enumerate}

\end{proposition}
\begin{proof}
The properties (a) and (b) are implied from \eqref{re3}. The properties (c) and (d) were proved in \cite{KT2022}.
\end{proof}

Now we look for a representation of the solution to the following initial value  linear problem
\begin{align}
\partial_t u - (1 + D_t^{\{m\}})\Delta u & = F \;\text{ in }\Omega, t\in (0,T],\label{le1}\\
 u & = 0\; \text{ on } \partial\Omega,\; t\in [0,T],\label{le2}\\
u(0) & = \xi \; \text{ in }\Omega, \label{le3}
\end{align}
where $F\in C([0,T];L^2(\Omega))$.

Assume that 
\begin{align*}
u(t) = \sum_{n=1}^\infty u_n(t)e_n, \; F(t) = \sum_{n=1}^\infty F_n(t)e_n.
\end{align*}
Using these expansions in \eqref{le1}, we obtain
\begin{align*}
&u_n'(t) +\lambda_n (1 + D_t^{\{m\}})u_n(t) = F_n(t),\\
&u_n(0) = \xi_n := (\xi,e_n).
\end{align*}
Applying Proposition \ref{pp-relax-func}(d), we get
$$
u_n(t) = \omega(t,\lambda_n)\xi_n + \int_0^t \omega(t-\tau,\lambda_n)F_n(\tau)d\tau.
$$
Therefore, 
\begin{align}
u(t) = S(t)\xi + \int_0^t S(t-\tau)F(\tau)d\tau,\label{sol-form}
\end{align}
where $S(t)$ is the \textit{resolvent operator} determined by
\begin{align}
S(t)\xi = \sum_{n=1}^\infty \omega(t,\lambda_n)\xi_n e_n ,\; \xi\in L^2(\Omega).\label{sol-op}
\end{align}
Obviously, $S(t)$ is a bounded linear operator on  $L^2(\Omega)$ for all $t\ge 0$. Moreover, we have the following statements.
\begin{lemma}\label{lm-sol-op}
Let $\{S(t)\}_{t\ge 0}$ be the resolvent family defined by \eqref{sol-op}, $v\in L^2(\Omega)$ and $T>0$. Then,
\begin{enumerate}
\item[(a)] $S(\cdot)v\in C([0,T];L^2(\Omega))$ and $\|S(t)\|\le \omega(t,\lambda_1)$ for all $t\ge 0$.
\item[(b)] For $g\in C([0,T];\mathbb H^{\mu-1})$, $\mu>0$, we have $S*g\in C([0,T];\mathbb H^{\mu})$. Furthermore,
\begin{equation}\label{lm-sol-op-a}
\|S*g(t)\|^2_{\mathbb H^\mu}\le \int_0^t \omega(t-\tau,\lambda_1)\|g(\tau)\|^2_{\mathbb H^{\mu-1}}d\tau,\text{ for all } t\ge 0.
\end{equation}
\item[(c)] If $m$ is nonincreasing, then $S(\cdot)v\in C^1((0,T];L^2(\Omega))$ and it holds that
\begin{equation}\label{lm-sol-op-b}
\|S'(t)\|\le t^{-1}\;\text{ for all } t>0.
\end{equation}
\item[(d)] For $\delta\in (0,1)$, $g\in C([0,T];\mathbb H^{\mu-1-\delta})$, we have
\begin{align*}
\|S*g(t)\|^2_{\mathbb H^\mu}\le \int_0^t (t-\tau)^{-\delta}\|g(\tau)\|^2_{\mathbb H^{\mu-1-\delta}}d\tau.
\end{align*}
\item[(e)] If  $(1*m)^{-1}\in L^1(0,T)$, then for  $g\in C([0,T];\mathbb H^{\mu-2})$, we have
\begin{align*}
\|S*g(t)\|^2_{\mathbb H^\mu}\le \int_0^t \frac{\|g(\tau)\|^2_{\mathbb H^{\mu-2}}}{(1*m)(t-\tau)}d\tau.
\end{align*}
\end{enumerate}
\end{lemma}
\begin{proof}
The properties (a) and (c) were proved in \cite[Lemma 1]{KT2022}. Now we show (b). For $g\in C([0,T];\mathbb H^{\mu-1})$, we have
\begin{align*}
\|S*g(t)\|^2_{\mathbb H^\mu} = \sum_{n=1}^\infty \lambda^\mu_n \left(\int_0^t \omega(t-\tau,\lambda_n)g_n(\tau)d\tau\right)^2,\; g_n(\tau)=(g(\tau),e_n).
\end{align*}
Moreover, by the H\"older inequality, we can estimate
\begin{align*}
\left(\int_0^t \omega(t-\tau,\lambda_n)g_n(\tau)d\tau\right)^2 & \le \left(\int_0^t \omega(t-\tau,\lambda_n)d\tau\right)\left(\int_0^t \omega(t-\tau,\lambda_n)|g_n(\tau)|^2d\tau\right)\\
& \le \lambda_n^{-1}\int_0^t \omega(t-\tau,\lambda_1)|g_n(\tau)|^2d\tau.
\end{align*}
Hence,
\begin{align*}
\|S*g(t)\|^2_{\mathbb H^\mu} & \le \sum_{n=1}^\infty \int_0^t \omega(t-\tau,\lambda_1)\lambda_n^{\mu-1}|g_n(\tau)|^2 d\tau\\
& = \int_0^t \omega(t-\tau,\lambda_1)\|g(\tau)\|^2_{\mathbb H^{\mu-1}} d\tau.
\end{align*}
Next, we prove (d). Assume that $g\in C([0,T];\mathbb H^{\mu-1-\delta})$. Then
\begin{align*}
\|S*g(t)\|^2_{\mathbb H^\mu}=\sum_{n=1}^\infty \lambda_n^\mu \left(\int_0^t \omega(t-\tau,\lambda_n)g_n(\tau)d\tau\right)^2,\; g_n(\tau)=(g(\tau),e_n).
\end{align*}
Using the H\"older inequality and Proposition \ref{pp-relax-func}, we obtain
\begin{align*}
 \left(\int_0^t \omega(t-\tau,\lambda_n)g_n(\tau)d\tau\right)^2 & \le \left(\int_0^t \omega(t-\tau,\lambda_n)d\tau\right)\left(\int_0^t \omega(t-\tau,\lambda_n)|g_n(\tau)|^2 d\tau\right)\\
 & \le \lambda_n^{-1}\int_0^t \frac{|g_n(\tau)|^2}{1+\lambda_n (t-\tau)} d\tau\\
 & \le \lambda_n^{-1}\int_0^t \frac{|g_n(\tau)|^2}{\lambda_n^\delta (t-\tau)^\delta} d\tau,
\end{align*}
here, we used the inequality $1+b\ge b^\delta$ for $b\ge 0, \delta\in (0,1)$.

Therefore,
\begin{align*}
\|S*g(t)\|^2_{\mathbb H^\mu}&\le \sum_{n=1}^\infty \lambda_n^{\mu-1-\delta}\int_0^t (t-\tau)^{-\delta}|g_n(\tau)|^2 d\tau\\
& = \int_0^t (t-\tau)^{-\delta}\|g(\tau)\|^2_{\mathbb H^{\mu-1-\delta}}d\tau.
\end{align*}
The last property is proved similarly by noting that
\begin{align*}
 \left(\int_0^t \omega(t-\tau,\lambda_n)g_n(\tau)d\tau\right)^2 & \le \left(\int_0^t \omega(t-\tau,\lambda_n)d\tau\right)\left(\int_0^t \omega(t-\tau,\lambda_n)|g_n(\tau)|^2 d\tau\right)\\
 & \le \lambda_n^{-1}\int_0^t \frac{|g_n(\tau)|^2}{1+\lambda_n (1*m)(t-\tau)} d\tau\\
 & \le \lambda_n^{-2}\int_0^t \frac{|g_n(\tau)|^2}{ (1*m)(t-\tau)} d\tau.
\end{align*}
The proof is complete.
\end{proof}

\section{Global solvability and  H\"older regularity}

In order to solve problem \eqref{e1}-\eqref{e3}, we require the following assumption on the nonlinearity:
\begin{enumerate}\it
\item[(F)] The function $f: \mathbb H^{\mu}\times \mathbb H^{\mu} \to \mathbb H^{-\theta}$ satisfies  $f(0,0)=0$ and for $\rho, \rho'>0$,  we have
$$
\|f(v_1,w_1)-f(v_2,w_2)\|_{\mathbb H^{-\theta}}\le L_f(\rho)\|v_1-v_2\|_{\mathbb H^\mu}+K_f(\rho')\|w_1-w_2\|_{\mathbb H^\mu}, 
$$
where  $\|v_1\|_{\mathbb H^\mu}, \|v_2\|_{\mathbb H^\mu}\le \rho$, $\|w_1\|_{\mathbb H^\mu}, \|w_2\|_{\mathbb H^\mu}\le \rho'$, $0<\mu<2$, $\theta >0$ , $L_f$ and $K_f$ are nonnegative functions.
\end{enumerate}
By the representation of the solution  to the linear problem given by \eqref{sol-form}, we have the following definition of mild solution to \eqref{e1}-\eqref{e3}.
\begin{definition}
Let $\xi\in\mathbb H^\mu$. The function $u\in C([0,T];\mathbb H^\mu)$ is called a mild solution to  \eqref{e1}-\eqref{e3} on the interval $[0,T]$ if the following  identity holds
$$
u(t) = S(t)\xi + \int_0^t S(t-\tau)f(u(\tau),\mathcal Hu(\tau))d\tau,\; t\in [0,T].
$$
\end{definition}
\subsection{Global solvability}
\begin{theorem}\label{th-sol}
Assume that the condtion (F) is satisfied with $\theta=1+\delta-\mu$, $\delta\in(0,1)$ and 
$$
\limsup\limits_{\rho\to 0}L_f(\rho)=L_f^*,\; \limsup\limits_{\rho'\to 0}K_f(\rho')=K_f^*,
$$
such that
$$
8T^{1-\delta}(1-\delta)^{-1}({L^*_f}^2+{K^*_f}^2\|\ell\|^2_{L^1})<1.
$$
Then there exists $\rho^*>0$ such that for $\|\xi\|_{\mathbb H^\mu}\le \frac 12 \rho^*$, problem \eqref{e1}-\eqref{e3} possesses a unique mild solution $u$ on $[0,T]$ obeying $\|u(t)\|_{\mathbb H^\mu}\le \rho^*$, for $t\in [0,T]$.
\end{theorem}
\begin{proof}
We will show that the operator $\Phi$ determined by
$$
\Phi(u)(t)=S(t)\xi + \int_0^t S(t-\tau)f(u(\tau),\mathcal Hu(\tau))d\tau,\; t\in [0,T],
$$
has a fixed point in the space $C([0,T];\mathbb H^\mu)$ furnished by the norm $\|u\|_\infty=\sup\limits_{t\in [0,T]}\|u(t)\|_{\mathbb H^\mu}$. 

Denote by $B_\rho$ the closed ball of radius $\rho$ in $C([0,T];\mathbb H^\mu)$ which centers at the origin.  For $u\in B_\rho$, we have $\mathcal Hu\in B_{\rho'}$ with $\rho' = \rho\|\ell\|_{L^1}$. Then
\begin{align*}
\|\Phi(u)(t)\|^2_{\mathbb H^\mu} & \le 2\|S(t)\xi\|^2_{\mathbb H^\mu} + 2\|S*f(u(\cdot),\mathcal Hu(\cdot))(t)\|^2_{\mathbb H^\mu}\\
& \le 2\|S(t)\xi\|^2_{\mathbb H^\mu} + 2\int_0^t (t-\tau)^{-\delta}\|f(u(\tau),\mathcal Hu(\tau))\|^2_{\mathbb H^{\mu-1-\delta}}d\tau,
\end{align*} 
by Lemma  \ref{lm-sol-op}(d). Using the condition (F), we have
\begin{align*}
\|\Phi(u)(t)\|^2_{\mathbb H^\mu} & \le 2\|\xi\|^2_{\mathbb H^\mu} + 4\int_0^t (t-\tau)^{-\delta}[L_f(\rho)^2 \|u(\tau)\|^2_{\mathbb H^\mu} +K_f(\rho')^2\|\mathcal Hu(\tau)\|^2_{\mathbb H^\mu}]d\tau\\
& \le 2\|\xi\|^2_{\mathbb H^\mu} + 4\int_0^t (t-\tau)^{-\delta}({L^*_f}^2+{K^*_f}^2\|\ell\|^2_{L^1}+\epsilon) \rho^2 d\tau\\
& \le 2\|\xi\|^2_{\mathbb H^\mu} + 4T^{1-\delta}(1-\delta)^{-1}({L^*_f}^2+{K^*_f}^2\|\ell\|^2_{L^1}+\epsilon) \rho^2,
\end{align*}
where $\epsilon>0$ and $\rho^*>0$ is chosen such that 
$$
8T^{1-\delta}(1-\delta)^{-1}({L^*_f}^2+{K^*_f}^2\|\ell\|^2_{L^1}+\epsilon)\le 1, \text{ for all } \rho\le \rho^*.
$$
Then, for $u\in B_{\rho^*}$ and $\|\xi\|\le \frac 12\rho^*$,  we can estimate
\begin{align*}
\|\Phi(u)(t)\|_{\mathbb H^\mu}\le \rho^*, \text{ for all } t\in [0,T].
\end{align*}
Hence $\Phi(B_{\rho^*})\subset B_{\rho^*}$. It remains to verify that $\Phi$ is a contraction on $B_{\rho^*}$. Indeed, for $u_1, u_2\in B_{\rho^*}$, we have
\begin{align*}
& \|f(u_1(\tau),\mathcal Hu_1(\tau))-f(u_2(\tau),\mathcal Hu_2(\tau))\|^2_{\mathbb H^{\mu-1-\delta}} \\
& \qquad \le 2 L_f(\rho^*)^2 \|u_1(\tau)-u_2(\tau)\|^2_{\mathbb H^\mu} + 2 K_f(\rho^*\|h\|_{L^1})^2\|\mathcal H u_1(\tau)-\mathcal Hu_2(\tau)\|^2_{\mathbb H^\mu}\\
& \qquad \le 2({L^*_f}^2+{K^*_f}^2\|h\|^2_{L^1}+\epsilon)\sup_{\tau\in [0,T]}\|u_1(\tau)-u_2(\tau)\|^2_{\mathbb H^\mu}.
\end{align*}
This yields that
\begin{align*}
& \|\Phi(u_1)(t)-\Phi(u_2)(t)\|^2_{\mathbb H^\mu} \le \|S*[f(u_1(\cdot),\mathcal Hu_1(\cdot))-f(u_2(\cdot),\mathcal Hu_2(\cdot))](t)\|^2_{\mathbb H^\mu}\\
&\qquad  \le \int_0^t (t-\tau)^{-\delta}\|f(u_1(\tau),\mathcal Hu_1(\tau))-f(u_2(\tau),\mathcal Hu_2(\tau))\|^2_{\mathbb H^{\mu-1-\delta}}d\tau\\
&\qquad  \le 2T^{1-\delta}(1-\delta)^{-1}({L^*_f}^2+{K^*_f}^2\|h\|^2_{L^1}+\epsilon) \sup_{\tau\in [0,T]} \|u_1(\tau)-u_2(\tau)\|^2_{\mathbb H^\mu}\\
&\qquad  \le \frac 14 \sup_{\tau\in [0,T]} \|u_1(\tau)-u_2(\tau)\|^2_{\mathbb H^\mu}.
\end{align*}
Finally, we arrive at
$$
\|\Phi(u_1)-\Phi(u_2)\|_\infty\le \frac 12 \|u_1-u_2\|_\infty.
$$
The proof is complete.
\end{proof}
In the case the function $f(\cdot,\cdot)$ satisfies global  Lipschitz condition, we have the following result.
\begin{theorem}\label{th-sol-ad}
Assume that the nonlinearity  $f:\mathbb H^{\mu}\times \mathbb H^{\mu}\to \mathbb H^{\mu-1-\delta}$ is continuous and satisfies the condition 
$$
\|f(v_1,w_1)-f(v_2,w_2)\|_{\mathbb H^{\mu-1-\delta}}\le L^*_f\|v_1-v_2\|_{\mathbb H^\mu}+K^*_f\|w_1-w_2\|_{\mathbb H^\mu}, 
$$
for all $v_1, v_2, w_1, w_2\in\mathbb H^\mu$, where  $L^*_f, K^*_f \ge 0$. Then, Problem  \eqref{e1}-\eqref{e3}  possesses a unique mild solution in the space $C([0,T];\mathbb H^\mu)$.
\end{theorem}
\begin{proof}
In $C([0,T];\mathbb H^\mu)$, we use an equivalent norm
$$
\|u\|_{\beta,\infty} = \sup_{t\in [0,T]}e^{-\beta t}\|u(t)\|_{\mathbb H^\mu},
$$
where $\beta>0$ satisfies
$$
L^*:=2({L^*_f}^2+{K^*_f}^2\|\ell\|^2_{L^1})  \int_0^T e^{-2\beta \tau} \tau^{-\delta}d\tau <1.
$$
Consider the solution operator $\Phi$ as in the proof of Theorem \ref{th-sol}, where $u_1, u_2\in C([0,T];\mathbb H^\mu)$, we have
\begin{align*}
\|\Phi(u_1)(t)-\Phi(u_2)(t)\|^2_{\mathbb H^\mu} & \le 2 \int_0^t (t-\tau)^{-\delta}{L^*_f}^2 \|u_1(\tau)-u_2(\tau)\|^2_{\mathbb H^\mu}d\tau\\
& \qquad + 2\int_0^t (t-\tau)^{-\delta}{K^*_f}^2\|\mathcal H u_1(\tau)-\mathcal H u_2(\tau)\|^2_{\mathbb H^\mu} d\tau,
\end{align*}
according to the assumption of the theorem and Lemma \ref{lm-sol-op}(d). Then, it yields
\begin{align*}
e^{-2\beta t}\|\Phi(u_1)(t)&-\Phi(u_2)(t)\|^2_{\mathbb H^\mu} \\
& \le 2\int_0^t e^{-2\beta(t-\tau)} (t-\tau)^{-\delta}{L^*_f}^2 [e^{-2\beta\tau}\|u_1(\tau)-u_2(\tau)\|^2_{\mathbb H^\mu}]d\tau\\
& \qquad + 2\int_0^t e^{-2\beta(t-\tau)} (t-\tau)^{-\delta}{K^*_f}^2 [e^{-2\beta\tau}\|\mathcal H u_1(\tau)-\mathcal H  u_2(\tau)\|^2_{\mathbb H^\mu}]d\tau\\
& \le 2({L^*_f}^2+{K^*_f}^2\|\ell\|^2_{L^1})\left(\int_0^t e^{-2\beta \tau} \tau^{-\delta}d\tau\right)\|u_1-u_2\|^2_{\beta,\infty},
\end{align*}
here, we have used the following estimate
\begin{align*}
e^{-2\beta\tau}\|\mathcal H u_1(\tau)-\mathcal H  u_2(\tau)\|^2_{\mathbb H^\mu} &\le \left(\int_0^\tau e^{-\beta(\tau-z)}\ell(\tau-z)[e^{-\beta z}\|u_1(z)-u_2(z)\|_{\mathbb H^\mu}]dz \right)^2\\
& \le \left(\int_0^\tau e^{-\beta z}|\ell (z)| dz \right)^2 \|u_1-u_2\|^2_{\beta,\infty}\\
& \le \|\ell\|^2_{L^1} \|u_1-u_2\|^2_{\beta,\infty}.
\end{align*}
Therefore,
$$
\|\Phi(u_1)-\Phi(u_2)\|_{\beta,\infty}\le \sqrt{L^*}\|u_1-u_2\|_{\beta,\infty}.
$$
In other words, $\Phi$ is a contraction on $C([0,T];\mathbb H^\mu)$. This completes the proof.
\end{proof}
\begin{remark}\label{rm-sol-ad}
(i) In Theorem \ref{th-sol-ad}, we do not require that $\|\xi\|_{\mathbb H^\mu}$ is small. Besides, we also relax the condition on $L^*_f$ and $K^*_f$ in comparison with the assumptions of Theorem \ref{th-sol}. Moreover, the result obtained in Theorem \ref{th-sol-ad} still holds if we add to the right-hand side of equation \eqref{e1} an external force, that means the right-hand side of   \eqref{e1} has the form $f(u,\mathcal Hu) + g(t,x)$.

(ii) In the case $(1*m)^{-1}\in L^1(0,T)$, we achieve a similar result to that in Theorem \ref{th-sol-ad} assumming that the nonlinearity $f$ can take \lq weaker\rq\  values:
$$
\|f(v_1,w_1)-f(v_2,w_2)\|_{\mathbb H^{\mu-2}}\le L^*_f\|v_1-v_2\|_{\mathbb H^\mu}+K^*_f\|w_1-w_2\|_{\mathbb H^\mu}, 
$$
for all $v_1, v_2, w_1, w_2\in\mathbb H^\mu$, where $L^*_f, K^*_f \ge 0$. In this situation, we use the estimate in Lemma \ref{lm-sol-op}(e).
\end{remark}
\subsection{H\"older regularity}

In this part, we assume an additional contition on the kernel $m$ as follows.
\begin{enumerate}\it
\item[\rm (M*)] The hypothesis (M) is satisfied with a nonincreasing function $m$.
\end{enumerate}
Note that, under the assumption (M*), the resolvent $S(\cdot)$ is differentiable in $(0,\infty)$ and $\|S(t)\|\le t^{-1}$ for $t>0$ by Lemma \ref{lm-sol-op}.

For $\gamma\in (0,1)$, denote
$$
V^{\mu,\gamma}_{\rho,\rho^*} = B_{\rho^*} \cap \{u\in C([0,T];\mathbb H^\mu): \sup_{\substack{h>0\\ t\in (0,T-h]}} \frac{t^\gamma\|u(t+h)-u(t)\|_{\mathbb H^\mu}}{h^\gamma}\le \rho\}.
$$
We will show that the mild solution to problem \eqref{e1}-\eqref{e3} obtained by Theorem \ref{th-sol} is H\"older continuous in $(0,T]$ by proving that the solution mapping $\Phi$ is contractive on $V^{\mu,\gamma}_{\rho,\rho^*}$.
\begin{theorem}\label{th-Holder}
Assume that (M*) and all assumptions in Theorem \ref{th-sol} are satisfied. Moreover, with  $\gamma\in (\frac 12 \delta, \frac 12)$, we have
\begin{align*}
&\ell^*_1 = \sup_{\substack{h>0\\ t\in (0,T-h]}}\left(\frac th\right)^\gamma \int_t^{t+h} |\ell(\tau)| d\tau<\infty,\\
&16 B(1-\delta, 1-2\gamma) T^{1-\delta} ( {L_f^*}^2 + {K_f^*}^2 {\ell^*_2}^2)<1,
\end{align*}
where $B(\cdot,\cdot)$ is the Beta function and
$$
\ell^*_2 = \sup_{t\in (0,T]}t^\gamma\int_0^t \frac{|\ell(\tau)|}{(t-\tau)^\gamma}d\tau.
$$
Then, the solution to problem \eqref{e1}-\eqref{e3} is H\"older continuous on $(0,T]$.
\end{theorem}
\begin{proof}
It suffices to show that $\Phi(V^{\mu,\gamma}_{\rho,\rho^*})\subset V^{\mu,\gamma}_{\rho,\rho^*}$ for a certain $\rho>0$.

Since $S(\cdot)$ is differentiable  in $(0,\infty)$, using the Mean value theorem, we have
\begin{align}
\|[S(t+h)-S(t)]\xi\|_{\mathbb H^\mu}&\le h\int_0^1\|S'(t+\zeta h)\xi\|_{\mathbb H^\mu}d\zeta\notag\\
& \le h\|\xi\|_{\mathbb H^\mu}\int_0^1\frac{d\zeta}{t+\zeta h} = \|\xi\|_{\mathbb H^\mu}\ln \left(1+\frac ht\right)\notag\\
& \le \|\xi\|_{\mathbb H^\mu}\gamma^{-1} t^{-\gamma} h^\gamma. \label{th-Holder-1}
\end{align}
On the other hand, for $u\in V^{\mu,\gamma}_{\rho,\rho^*}$, we have
\begin{align*}
& \|\mathcal H u(t+h)-\mathcal H u(t)\|_{\mathbb H^\mu} \le \int_t^{t+h} |\ell(\tau)|\|u(t+h-\tau)\|_{\mathbb H^\mu}d\tau \\
& \quad +\int_0^t |\ell(\tau)|\| u(t+h-\tau)-u(t-\tau)\|_{\mathbb H^\mu}d\tau\\
& \le t^{-\gamma}h^\gamma \rho^* \left(\frac th\right)^\gamma \int_t^{t+h} |\ell(\tau)| d\tau \\
& \quad + t^{-\gamma}h^\gamma t^\gamma\int_0^t \frac{|\ell(\tau)|}{(t-\tau)^\gamma} [(t-\tau)^\gamma h^{-\gamma}\| u(t+h-\tau)-u(t-\tau)\|_{\mathbb H^\mu}]d\tau\\
& \le t^{-\gamma}h^\gamma \rho^*\left[ \left(\frac th\right)^\gamma \int_t^{t+h} |\ell(\tau)| d\tau \right]
+ t^{-\gamma}h^\gamma \rho \left[t^\gamma\int_0^t \frac{|\ell(\tau)|}{(t-\tau)^\gamma}d\tau\right].
\end{align*}
Hence,
\begin{equation}\label{th-Holder-2}
\|\mathcal H u(t+h)-\mathcal H u(t)\|_{\mathbb H^\mu} \le t^{-\gamma}h^\gamma (\rho^* \ell^*_1 + \rho\ell^*_2).
\end{equation}
Denote
\begin{equation}\label{Dh}
D_h f(u)(t) = f(u(t+h),\mathcal H u(t+h))-f(u(t),\mathcal Hu(t)).
\end{equation}
We have
\begin{align}
& \|D_h f(u)(t)\|_{\mathbb H^{\mu-1-\delta}} \le \|f(u(t+h),\mathcal H u(t+h)) - f(u(t),\mathcal H u(t))\|_{\mathbb H^{\mu-1-\delta}}\notag\\
&\quad \le L_f(\rho^*)\|u(t+h)-u(t)\|_{\mathbb H^\mu}+K_f(\rho^*\|\ell\|_{L^1})\|\mathcal H u(t+h)-\mathcal H u(t)\|_{\mathbb H^\mu}\notag\\
&\quad  \le t^{-\gamma} h^\gamma L_f(\rho^*)\rho  + t^{-\gamma}h^\gamma K_f(\rho^*\|\ell\|_{L^1}) (\rho^* \ell^*_1 + \rho\ell^*_2),\label{th-Holder-3}
\end{align}
here, we employed the estimate \eqref{th-Holder-2}.

Finally, it holds that
\begin{align*}
\|\Phi(u)(t+h)&-\Phi(u)(t)\|^2_{\mathbb H^\mu}\le 2\|[S(t+h)-S(t)]\xi\|^2_{\mathbb H^\mu} \\
&+ 4\int_0^t \tau^{-\delta}\|D_h f(u)(t-\tau)\|^2_{\mathbb H^{\mu-1-\delta}}d\tau\\
&+ 4\int_t^{t+h}\tau^{-\delta}\|f(u(t+h-\tau),\mathcal H u(t+h-\tau) )\|^2_{\mathbb H^{\mu-1-\delta}}d\tau \\
& = E_1(t) + E_2(t)+E_3(t).
\end{align*}
By \eqref{th-Holder-1}, we get
\begin{equation*}
E_1(t) \le 2\gamma^{-2}t^{-2\gamma}h^{2\gamma} \|\xi\|^2_{\mathbb H^\mu}.
\end{equation*}
Using \eqref{th-Holder-3}, we gain
\begin{align*}
E_2(t) & \le  8 h^{2\gamma}[\rho^2 L_f(\rho^*)^2 + K_f(\rho^*\|\ell\|^2_{L^1})^2(\rho^* \ell^*_1 + \rho\ell^*_2)^2] \int_0^t \tau^{-\delta}(t-\tau)^{-2\gamma} d\tau\\
& \le 8 h^{2\gamma}t^{-2\gamma}[\rho^2 L_f(\rho^*)^2 + K_f(\rho^*\|\ell\|^2_{L^1})^2(\rho^* \ell^*_1 + \rho\ell^*_2)^2] B(1-\delta,1-2\gamma) T^{1-\delta}\\
& \le 16 h^{2\gamma}t^{-2\gamma} B(1-\delta,1-2\gamma) T^{1-\delta} \\
& \quad \times [ (L_f(\rho^*)^2 + K_f(\rho^*\|\ell\|^2_{L^1})^2 {\ell^*_2}^2)\rho^2
+K_f(\rho^*\|\ell\|^2_{L^1})^2 {\rho^*}^2 {\ell^*_1}^2 ]   
\end{align*}
here we utilized the identity
\begin{align*}
\int_0^t \tau^{-\delta}(t-\tau)^{-2\gamma} d\tau =B(1-\delta,1-2\gamma) t^{1-\delta-2\gamma}.
\end{align*}
We can estimate  $E_3(t)$ as follows
\begin{align*}
E_3(t) & = 4\int_0^h (t+h-\tau)^{-\delta}\|f(u(\tau),\mathcal H u(\tau) )\|^2_{\mathbb H^{\mu-1-\delta}}d\tau\\
& \le 8{\rho^*}^2 [L_f(\rho^*)^2 + K_f(\rho^*\|\ell\|_{L^1})^2\|\ell\|^2_{L^1}]\int_0^h (t+h-\tau)^{-\delta}d\tau.
\end{align*}
Noting that
\begin{align*}
\int_0^h (t+h-\tau)^{-\delta}d\tau & = (1-\delta)^{-1}[(t+h)^{1-\delta}-t^{1-\delta}] \\
& \le h t^{-\delta}\le h^{2\gamma} t^{-2\gamma} h^{1-2\gamma}t^{2\gamma-\delta}\le h^{2\gamma} t^{-2\gamma} T^{1-\delta},
\end{align*}
we obtain
\begin{equation*}
E_3(t) \le 8h^{2\gamma} t^{-2\gamma} (\rho^*)^2 T^{1-\delta} [L_f(\rho^*)^2 + K_f(\rho^*\|\ell\|_{L^1})^2\|\ell\|^2_{L^1}].
\end{equation*}
Combining the estimates of $E_1(t), E_2(t)$ and $E_3(t)$ above, we conclude that
\begin{align*}
\left(\frac th\right)^{2\gamma}& \|\Phi(u)(t+h)-\Phi(u)(t)\|^2_{\mathbb H^\mu} \le 2\gamma^{-2} \|\xi\|^2_{\mathbb H^\mu}\\
& + 16 B(1-\delta,1-2\gamma) T^{1-\delta} [ ({L_f^*}^2 + {K_f^*}^2 {\ell^*_2}^2+\epsilon)\rho^2
+K_f(\rho^*\|\ell\|^2_{L^1})^2 {\rho^*}^2 {\ell^*_1}^2 ]\\
& + 8 {\rho^*}^2 T^{1-\delta} [L_f(\rho^*)^2 + K_f(\rho^*\|\ell\|_{L^1})^2\|\ell\|^2_{L^1}],
\end{align*}
here, $\epsilon>0$ is chosen such that
$$
16 T^{1-\delta}[L_f(\rho^*)^2 + K_f(\rho^*\|\ell\|^2_{L^1})^2 {\ell^*_2}^2]\le 16 T^{1-\delta}({L_f^*}^2 + {K_f^*}^2 {\ell^*_2}^2+\epsilon)< \frac1{B(1-\delta,1-2\gamma)}.
$$
Now, we take a sufficiently large  $\rho> 0$ such that
$$
\left(\frac th\right)^{2\gamma} \|\Phi(u)(t+h)-\Phi(u)(t)\|^2_{\mathbb H^\mu} \le \rho^2, \text{ for all } h>0, t\in (0,T-h].
$$
This implies that $\Phi(u)\in W^{\mu,\gamma}_{\rho,\rho^*}$. The proof is complete.
\end{proof}
We now consider the case when the nonlinearity $f(u(t),\mathcal Hu(t))$ is more regular, i.e. it takes values in $\mathbb H^{\mu-1}$. We will show the H\"older continuity of the solution without the assumptions on the coefficients as in Theorem \ref{th-Holder}.
\begin{theorem}\label{th-solb}
Assume that the condition (M*) and (F) are satisfied with $\theta=1-\mu$, and
$$
\limsup\limits_{\rho\to 0}L_f(\rho)=L_f^*,\; \limsup\limits_{\rho'\to 0}K_f(\rho')=K_f^*,
$$
such that
$$
4({L^*_f}^2+{K^*_f}^2\|\ell\|^2_{L^1})<\lambda_1.
$$
Then, there exists $\eta>0$ such that for $\|\xi\|_{\mathbb H^\mu}\le \eta$, problem \eqref{e1}-\eqref{e3} has a unique mild solution $u$ in $[0,T]$. Moreover, $u(\cdot)$ is H\"older continuous on $(0,T]$.
\end{theorem}
\begin{proof}
Similar to the proof of Theorem \ref{th-sol}, first we look for $\rho^*>0$ such that $\Phi(B_{\rho^*})\subset  B_{\rho^*}$. 
Let  $u\in B_\rho$, $\rho>0$, and $f(u(t),\mathcal Hu(t))\in \mathbb H^{\mu-1}$. 

Then using Lemma \ref{lm-sol-op}, we obtain
\begin{align*}
\|S*f(u(\cdot),&\mathcal Hu(\cdot))\|^2_{\mathbb H^\mu}  \le \int_0^t \omega(t-\tau,\lambda_1)\|f(u(\tau),\mathcal Hu(\tau))\|^2_{\mathbb H^{\mu-1}}d\tau\\
& \le 2\int_0^t \omega(t-\tau,\lambda_1) [L_f(\rho)^2 \| u(\tau)\|^2_{\mathbb H^\mu} + K_f(\rho\|\ell\|_{L^1})^2 \| \mathcal H u(\tau)\|^2_{\mathbb H^\mu}]d\tau\\
& \le 2\int_0^t \omega(t-\tau,\lambda_1) [L_f(\rho)^2  + K_f(\rho\|\ell\|_{L^1})^2\|\ell\|^2_{L^1} ] \| u(\tau)\|^2_{\mathbb H^\mu}d\tau\\
& \le 2\int_0^t \omega(t-\tau,\lambda_1) ({L_f^*}^2  + {K_f^*}^2\|\ell\|_{L^1}^2+\epsilon ) \| u(\tau)\|^2_{\mathbb H^\mu}d\tau,
\end{align*}
for $\epsilon>0$ and $\rho^*>0$ is chosen such that 
$$
4[L_f(\rho)^2  + K_f(\rho\|\ell\|_{L^1})^2\|\ell\|^2_{L^1}]\le 4({L_f^*}^2  + {K_f^*}^2\|\ell\|_{L^1}^2+\epsilon) \le \lambda_1,
$$
for all $\rho\le \rho^*$. Then,
\begin{align*}
\|S*f(u(\cdot),&\mathcal Hu(\cdot))\|^2_{\mathbb H^\mu}  \le 2\lambda_1^{-1}(1-\omega(t,\lambda_1)) ({L_f^*}^2  + {K_f^*}^2\|\ell\|_{L^1}^2+\epsilon){\rho^*}^2.
\end{align*}
We can estimate the solution operator $\Phi$ as follows
\begin{align*}
\|\Phi(u)(t)\|_{\mathbb H^\mu}^2&\le 2\omega(t,\lambda_1)^2\|\xi\|^2_{\mathbb H^\mu} + 2\|S*f(u(\cdot),\mathcal Hu(\cdot))\|^2_{\mathbb H^\mu}\\
& \le 2\omega(t,\lambda_1) \|\xi\|^2_{\mathbb H^\mu}+4 \lambda_1^{-1}(1-\omega(t,\lambda_1)) ({L_f^*}^2  + {K_f^*}^2\|\ell\|_{L^1}^2+\epsilon){\rho^*}^2\\
& \le 2\omega(t,\lambda_1)[\|\xi\|^2_{\mathbb H^\mu}-2 \lambda_1^{-1}({L_f^*}^2  + {K_f^*}^2\|\ell\|_{L^1}^2+\epsilon){\rho^*}^2]\\
& \qquad + 4 \lambda_1^{-1}({L_f^*}^2  + {K_f^*}^2\|\ell\|_{L^1}^2+\epsilon){\rho^*}^2.
\end{align*}
Take $\eta=2 \lambda_1^{-1}({L_f^*}^2  + {K_f^*}^2\|\ell\|_{L^1}^2+\epsilon){\rho^*}^2$, with $\|\xi\|_{\mathbb H^\mu}\le \eta$ we have
\begin{align*}
\|\Phi(u)(t)\|_{\mathbb H^\mu}^2&\le 4 \lambda_1^{-1}({L_f^*}^2  + {K_f^*}^2\|\ell\|_{L^1}^2+\epsilon){\rho^*}^2 \le {\rho^*}^2.
\end{align*}
Hence $\Phi(u)\in B_{\rho^*}$. Then, we show that $\Phi$ is contractive on $B_{\rho^*}$. For $u_1, u_1\in B_{\rho^*}$, we have
\begin{align*}
\|\Phi(u_1)(t)&-\Phi(u_2)(t)\|_{\mathbb H^\mu}^2 \\
& \le \int_0^t \omega(t-\tau,\lambda_1)\|f(u_1(\tau),\mathcal Hu_1(\tau))-f(u_2(\tau),\mathcal Hu_2(\tau))\|^2_{\mathbb H^{\mu-1}}d\tau\\
& \le  \int_0^t \omega(t-\tau,\lambda_1) ({L_f^*}^2+{K_f^*}^2\|\ell\|^2_{L^1}+\epsilon)\|u_1(\tau)-u_2(\tau)\|^2_{\mathbb H^\mu}d\tau\\
& \le \lambda_1^{-1}({L_f^*}^2+{K_f^*}^2\|\ell\|^2_{L^1}+\epsilon)\sup_{\tau\in [0,T]}\|u_1(\tau)-u_2(\tau)\|^2_{\mathbb H^\mu}\\
& \le \frac 14 \sup_{\tau\in [0,T]}\|u_1(\tau)-u_2(\tau)\|^2_{\mathbb H^\mu}.
\end{align*}
Therefore
$$
\|\Phi(u_1)-\Phi(u_2)|_\infty\le \frac 12 \|u_1-u_2\|_\infty,
$$
which implies that $\Phi$ is contractive on $B_{\rho^*}$, and problem \eqref{e1}-\eqref{e3} possesses a unique mild solution $u\in B_{\rho^*}$.

It remains to prove that this solution is  H\"older continuous on $(0, T]$. Observe that
\begin{align*}
u(t+h)-u(t) &= [S(t+h)-S(t)]\xi + \int_0^t S(\tau) D_h f(u) (t-\tau)d\tau\\
& + \int_t^{t+h} S(\tau) f(u(t+h-\tau),\mathcal Hu(t+h-\tau)) d\tau.
\end{align*}
Then
\begin{align*}
\|u(t+h)-u(t)\|^2_{\mathbb H^\mu}&\le 3 \|[S(t+h)-S(t)]\xi\|^2_{\mathbb H^\mu} \\
& + 3\int_0^t \|D_h f(u)(t-\tau)\|^2_{\mathbb H^{\mu-1}}d\tau\\
& + 3\int_t^{t+h}\|f(u(t+h-\tau),\mathcal Hu(t+h-\tau))\|_{\mathbb H^{\mu-1}}^2d\tau\\
& = F_1(t) + F_2(t) + F_3(t),
\end{align*}
here we employed Lemma \ref{lm-sol-op}(b) and the fact that $\omega(t,\lambda_1)\le 1$.

By an analoguous argument as in the proof of Theorem \ref{th-Holder}, we get
\begin{align*}
\|[S(t+h)-S(t)]\xi\|_{\mathbb H^\mu}& \le \gamma^{-1} h^{\gamma} t^{-\gamma}\|\xi\|_{\mathbb H^\mu}, \gamma\in (0,\frac 12).
\end{align*}
So it is clear that
$$
F_1(t) \le 3 \gamma^{-2} h^{2\gamma} t^{-2\gamma}\|\xi\|_{\mathbb H^\mu}^2.
$$
Moreover,
\begin{align*}
F_2(t) & \le 6 \int_0^t ({L_f^*}^2+{K_f^*}^2\|\ell\|^2_{L^1}+\epsilon)\|u(t+h-\tau)-u(t-\tau)\|^2_{\mathbb H^\mu}d\tau\\
& = 6 \int_0^t ({L_f^*}^2+{K_f^*}^2\|\ell\|^2_{L^1}+\epsilon)\|u(\tau+h)-u(\tau)\|^2_{\mathbb H^\mu}d\tau,
\end{align*}
\begin{align*}
F_3(t)&\le 6 \int_t^{t+h}  ({L_f^*}^2+{K_f^*}^2\|\ell\|^2_{L^1}+\epsilon){\rho^*}^2 d\tau\\
& \le 6h({L_f^*}^2+{K_f^*}^2\|\ell\|^2_{L^1}+\epsilon){\rho^*}^2\\
& \le 6T h^{2\gamma}t^{-2\gamma} ({L_f^*}^2+{K_f^*}^2\|\ell\|^2_{L^1}+\epsilon){\rho^*}^2,
\end{align*}
thanks to the fact that $0<h, t<T$. 
Putting the estimates of $F_1(t), F_2(t)$ and $F_3(t)$ together, we get
\begin{align*}
\|u(t+h)&-u(t)\|^2_{\mathbb H^\mu}
 \le C_0 h^{2\gamma} t^{-2\gamma}  +  \int_0^t C_1\tau^{-2\gamma} [\tau^{2\gamma}\|u(\tau+h)-u(\tau)\|^2_{\mathbb H^\mu}]d\tau
\end{align*}
with 
\begin{align*}
C_0 &= 3\gamma^{-2} \|\xi\|^2_{\mathbb H^\mu}+6T ({L_f^*}^2+{K_f^*}^2\|\ell\|^2_{L^1}+\epsilon){\rho^*}^2,\\
C_1 & = 6({L_f^*}^2+{K_f^*}^2\|\ell\|^2_{L^1}+\epsilon).
\end{align*}
Applying the Gronwall inequality, we obtain
\begin{align*}
t^{2\gamma}\|u(t+h)&-u(t)\|^2_{\mathbb H^\mu} \le e^{C_1(1-2\gamma)^{-1} t^{1-2\gamma}} C_0 h^{2\gamma}\text{ for } t>0.
\end{align*}
In other words, $u(\cdot)$ is  H\"older continuous on $(0,T]$.
\end{proof}
\begin{example}\rm 
Consider the case $d\ge 3$ and
$$
f(u(t),\mathcal Hu(t)) = |u(t)|^p +  \int_0^t \ell (t-\tau) (\chi\cdot\nabla) u(\tau) d\tau, p>1,
$$
with $\chi\in [L^\infty (\Omega)]^d$. The nonlinear function contains the advection term which depends on the history of states.  Noting that, for $\theta\in (0,1)$, $q=\frac{2d}{d+2\theta}$,  $\hat p=\frac{2d}{d-2}$, $\hat q=\frac{2d}{2+2\theta}$, we have
$$
\frac{1}{\hat p} +\frac{1}{\hat q}=\frac 1q.
$$
Applying general H\"older inequality with $u\in L^{(p-1)\hat q}(\Omega), v\in L^{\hat p}(\Omega)$, we gain
\begin{align*}
\||u|^{p-1}v\|_{L^q} &\le \||u|^{p-1}\|_{L^{\hat q}}\|v\|_{L^{\hat p}} \\
& = \|u\|^{p-1}_{L^{(p-1)\hat q}}\|v\|_{L^{\hat p}}.
\end{align*}
Assume that $(p-1)\hat q\le \hat p$, which means
$$
p\le \frac{d+2\theta}{d-2}.
$$
By Lemma \ref{lm:em1}, we have
$$
\mathbb H^1\subset H_0^1(\Omega)\subset L^{\hat p}(\Omega)\subset L^{(p-1)\hat q}(\Omega).
$$
Therefore
\begin{align*}
\||u|^{p-1}v\|_{L^q} \le C \|u\|^{p-1}_{\mathbb H^1}\|v\|_{\mathbb H^1},
\end{align*}
where $C$ is a positive constant which does not depend on $u$ and $v$. Moreover, we have $$L^q(\Omega) \subset H^{-\theta}\subset \mathbb H^{-\theta}.$$ Hence,
\begin{align}
\||u|^{p-1}v\|_{\mathbb H^{-\theta}} \le C \|u\|^{p-1}_{\mathbb H^1}\|v\|_{\mathbb H^1}.\label{ex-1}
\end{align}
When $u,v\in C([0,T];\mathbb H^1)$, applying the inequality  \eqref{ex-1}, we have
\begin{align}\label{ex-2}
\||u(t)|^p - |v(t)|^p\|_{\mathbb H^{-\theta}} \le C (\|u(t)\|^{p-1}_{\mathbb H^1}+\|v(t)\|^{p-1}_{\mathbb H^1})\|u(t)-v(t)\|_{\mathbb H^1}.
\end{align}
For $\mathcal Hu (t) = \ell*u$, we have
$$
\|(\chi\cdot\nabla)[\mathcal H u(t) - \mathcal H v(t)]\| \le \|\chi\|_\infty \cdot\|\mathcal H u(t) - \mathcal H v(t)\|_{\mathbb H^1}.
$$
Because $L^2(\Omega)\subset \mathbb H^{-\theta}$, we get
\begin{align}\label{ex-3}
\|(\chi\cdot\nabla)[\mathcal H u(t) - \mathcal H v(t)]\|_{\mathbb H^{-\theta}} \le C_\theta\|\chi\|_\infty \cdot\|\mathcal H u(t) - \mathcal H v(t)\|_{\mathbb H^1},
\end{align}
with $C_\theta$ is a positive constant. Combining \eqref{ex-2} and \eqref{ex-3}, it leads to
\begin{align*}
\|f(u(t),\mathcal Hu(t))-f(v(t),\mathcal Hv(t)\|_{\mathbb H^{-\theta}} & \le C (\|u(t)\|^{p-1}_{\mathbb H^1}+\|v(t)\|^{p-1}_{\mathbb H^1})\|u(t)-v(t)\|_{\mathbb H^1} \\
& \quad + C_\theta\|\chi\|_\infty \cdot\|\mathcal H u(t) - \mathcal H v(t)\|_{\mathbb H^1}.
\end{align*}
Thus, the function $f$ satisfies condition (F) with $\mu=1$, $\theta=\delta$ and
\begin{align*}
L_f(\rho) = 2C\rho^{p-1}, \; K_f(\rho)  = C_\theta\|\chi\|_\infty .
\end{align*}
Then, we get the conclusion of Theorem \ref{th-sol} và \ref{th-Holder} when $\|\chi\|_\infty$ is sufficiently small.
\end{example}

\section{Application}
Consider the following problem of identifying parameter 
\begin{align}
\partial_t u - (1+D_t^{\{m\}})\Delta u & = g(x)p(t) + f_1(u) \text{ in } \Omega, t\in (0,T),\label{pe1}\\
u & = 0 \text{ on } \partial\Omega,\; t\ge 0,\label{pe2}\\
u(0) & = \xi \text{ in } \Omega, \label{pe3}\\
\int_\Omega \kappa(x)u(t,x)dx & = \psi(t),\; t\in [0,T],\label{pe4}
\end{align}
where $p(t)$, $t\in [0,T]$, is an unknown parameter, $g\in L^2(\Omega)$ is given. In this model, \eqref{pe4} is the complementary measurement with $\kappa\in H_0^1(\Omega)$ such that $(g,\kappa)\ne 0$, $\psi\in W^{1,1}(0,T)$. 

In this problem, we require further that the kernel $m$ fulfils:
\begin{enumerate}\it
\item[\rm (M$^\star$)] Assumption (M) holds and $m'\in L^1(0,T)$.
\end{enumerate}
Then, problem \eqref{pe1} is rewritten as follows
\begin{align*}
\partial_t u - (1+m_0) \Delta u - m_1*\Delta u & = g(x)p(t) + f_1(u), \;m_0=m(0), m_1=m'.
\end{align*}
Combining with \eqref{pe4}, we obtain
$$
\psi' + (1+m_0)(\nabla u, \nabla\kappa) + m_1*(\nabla u, \nabla\kappa) = (g,\kappa)p(t) + (f_1(u),\kappa).
$$
Therefore,
$$
p(t) = (g,\kappa)^{-1}[\psi' + (1+m_0)(\nabla u, \nabla\kappa) + m_1*(\nabla u, \nabla\kappa)-(f_1(u),\kappa)].
$$
Set
$$
f_2(u,\mathcal H u):=(1+m_0)(\nabla u, \nabla\kappa) + (\nabla (m_1*u), \nabla\kappa)-(f_1(u),\kappa),
$$
where $\mathcal H$ is the convolution operator with the kernel $m_1$.

Assume that $f_1:\mathbb H^1\to L^2(\Omega)$ verifies the condition
$$
\|f_1(u)-f_1(v)\| \le L_1 \|u-v\|_{\mathbb H^1}, \text{ for all } u,v\in\mathbb H^1.
$$
Then, for all $u,v\in  C([0,T];\mathbb H^1)$, we have
\begin{align*}
|f_2(u,\mathcal H u)-f_2(v,\mathcal H v)| & \le (1+m_0)\|\kappa\|_{\mathbb H^1}\|u-v\|_{\mathbb H^1}\\
& \qquad + \|m_1*(u-v)\|_{\mathbb H^1}\|\kappa\|_{\mathbb H^1} + L_1\|\kappa\|\|u-v\|_{\mathbb H^1}.
\end{align*}
Hence,
$f(u,\mathcal Hu):= g (g,\kappa)^{-1} f_2(u,\mathcal Hu) + f_1(u)$ verifies the conditions of Theorem  \ref{th-sol-ad} with
\begin{align*}
L_f^*  &=|(g,\kappa)^{-1}|  \|g\| [(1+m_0)\|\kappa\|_{\mathbb H^1} + L_1\|\kappa\|] + L_1,\\
K_f^* &=  |(g,\kappa)^{-1}|  \|g\|\|\kappa\|_{\mathbb H^1}.
\end{align*}
Therefore, problem \eqref{pe1}-\eqref{pe4} is solvable in the sense that there exists a unique mild solution $(u,p)\in C([0,T];\mathbb H^1)\times C([0,T])$.

\end{document}